\documentclass[12pt]{amsart}
\usepackage[all]{xypic}
\usepackage{amsmath,graphics}
\usepackage{amsfonts,amssymb,hyperref}

\theoremstyle{plain}
\newtheorem*{theorem*}{Theorem}
\newtheorem*{lemma*} {Lemma}
\newtheorem*{corollary*} {Corollary}
\newtheorem*{proposition*} {Proposition}
\newtheorem{theorem}{Theorem}[section]
\newtheorem{lemma}[theorem]{Lemma}

\newtheorem{proposition}[theorem]{Proposition}

\theoremstyle{remark}
\newtheorem*{remark}{Remark}

\newtheorem*{claim}{Claim}

\theoremstyle{definition}

\def\be{\begin{equation}}
\def\ee{\end{equation}}

\def \K {\mathbf{K}}
\def \Z {\mathbf{Z}}

\def \G {\mathbf{G}}

\def\zg{\Z[\G]}

\def\R{\Bbb{R}}

\newcommand{\smfrac}[2]{\mbox{\footnotesize$\displaystyle\frac{#1}{#2}$}} 
\newcommand{\tmfrac}[2]{\mbox{\large$\frac{#1}{#2}$}} 

\def\eps{\epsilon}

\def\Q{\Bbb{Q}}

\def\K{\Bbb{K}}

\def\id{\op{id}}

\def\Z{\Bbb{Z}}

\def\N{\Bbb{N}}

\def\part{\partial}
\def\ll{\langle}
\def\rr{\rangle}

\def\g{\gamma}

\def\bp{\begin{pmatrix}}

\def\sm{\setminus}
\def\ep{\end{pmatrix}}
\def\bn{\begin{enumerate}}

\def\en{\end{enumerate}}
\def\ba{\begin{array}}
\def\ea{\end{array}}

\def\S{\Sigma}

\def\wti{\tilde}

\def\fr12{\frac{1}{2}}

\def\im{\op{Im}}

\def\op{\operatorname}

\def\ker{\op{ker}}

\def\hom{\op{Hom}}

\def\deg{\op{deg}}

\def\zt{\Z\tpm}

\def\ztt{\Z[t^{-1},t\doubler}

\def\zttleft{\Z\doublel t^{-1},t]}

\def\f{\Bbb{F}}

\def\K{\Bbb{K}}

\def\G{\Gamma}
\def\ol{\overline}

\def\ab{\op{ab}}

\newcommand{\smsum}[2]{\mbox{\footnotesize$\displaystyle\sum\limits_{#1}^{#2}$}} 
\newcommand{\tmsum}[2]{\mbox{$\textstyle \sum\limits_{#1}^{#2}$}} 

\def\zt{\Z[t^{\pm 1}]}

\def\K{\Bbb{K}}

\def\i{\mathfrak{I}}

\def\zgphi{\Z[\G_\phi]}
\def\zphig{\Z_{\phi}[\G\doubler}
\def\zphipi{\Z_{\phi}[\pi\doubler}
\def\zpsig{\Z_{\psi}[\G\doubler}

\def\ztt{\Z[t^{-1},t\doubler}

\def\i{\iota}

\def\cmtbf#1{} \def\cmt#1{}

\def\doubler{]\hspace{-0.05cm}]}
\def\doublel{[\hspace{-0.05cm}[}

\def\zphigright{\Z_{\phi}[\G\doubler}

\def\zphipiright{\Z_{\phi}[\pi\doubler}

\def\gphi{\G_{\phi}}

\begin{document}

\title{Novikov homology and noncommutative Alexander polynomials}

\author{Stefan Friedl}
\address{Fakult\"at f\"ur Mathematik\\ Universit\"at Regensburg\\   Germany}
\email{sfriedl@gmail.com}
\date{\today}

\begin{abstract}
In the early 2000's  Cochran and Harvey introduced non-commutative Alexander polynomials for 3-manifolds. Their degrees give strong lower bounds on the Thurston norm. In this paper we make the case that the vanishing of a certain Novikov--Sikorav homology module is the correct notion of a monic non-commutative Alexander polynomial. Furthermore we will use the opportunity to give new proofs of several statements about Novikov--Sikorav homology in the three-dimensional context.
\end{abstract}
\maketitle

\begin{center}
\emph{In memory of Tim Cochran.}\\[1cm]
\end{center}

\section{Summary of results}
Let $N$ be a 3-manifold. (Throughout this paper all 3-manifolds are understood to be connected, orientable, compact with empty or toroidal boundary.)
A class $\phi\in H^1(N;\Z)=\hom(\pi_1(N),\Z)$ is called \emph{fibered} if there exists a fibration $p\colon N\to S^1$ such that $p_*=\phi\colon \pi_1(N)\to \pi_1(S^1)=\Z$. A rational cohomology class $\phi$ in $H^1(N;\Q)$ is called fibered, if there exists an $n\in \N$ such that $n\phi$ is a fibered class in $H^1(N;\Z)$.

The complexity of  a surface $\S$ with connected components $\S_1,\dots,\S_k$ 
is defined to be
\[ \chi_-(\S):=\sum_{i=1}^d \max\{-\chi(\S_i),0\}.\]
Given a 3-manifold $N$ and $\phi \in H^1(N;\Z)$ the \emph{Thurston norm  of $\phi$}    is defined as
\[ \|\phi\|_T:= \min\{\chi_-(\S)\, |\, \S \subset N\mbox{ properly embedded and dual to }\phi\}.\]
Thurston~\cite{Th86} showed that this function defines a seminorm on $H^1(N;\Z)$. An elementary argument shows that it extends to a seminorm $x_N$ on rational cohomology $H^1(N;\Q)$.
We say that a properly embedded oriented surface $\Sigma$ in $N$ is \emph{Thurston norm minimizing} if the following hold:
\bn
\item $\chi_-(\Sigma)=\|\op{PD}([\Sigma])\|_T$,
\item there is no non-empty collection of components of $\Sigma$ that is (with the given orientation) null-homologous.
\en
If (1) is satisfied and if $\Sigma$ does not contain any components with non-negative Euler characteristic, then (2) is trivially satisfied.

The following theorem is a well-known fiberedness criterion.
The equivalence of the first two statements is essentially a consequence of Stallings' theorem~\cite{St61}, the resolution of the Poincar\'e Conjecture by Perelman and some straightforward group theory. The equivalence of the second and the third statement are also well-known, in particular it is a straightforward consequence of Proposition~\ref{prop:surface-connected} and \cite[Lemma~5.1]{EL83}.

\begin{theorem}\label{thm:fiberedness-condition}
Let $N$ be a 3-manifold and let $\phi\in H^1(N;\Z)$ be a primitive class. Then the following are equivalent:
\bn
\item $\phi$ is fibered,
\item there exists a connected  surface $\S$ dual to $\phi$ such that the two inclusion induced maps
\[ \iota_\pm\colon \pi_1(\S)\to \pi_1(N\sm \S\times (-1,1))\]
are isomorphisms, 
\item  any  Thurston norm minimizing surface $\S$ dual to $\phi$ is connected and for any such $\S$  the two inclusion induced maps
\[ \iota_\pm\colon \pi_1(\S)\to \pi_1(N\sm \S\times (-1,1))\]
are isomorphisms.
\en
\end{theorem}

Given a 3-manifold $N$ and $\phi\in H^1(N;\Z)=\hom(\pi_1(N),\ll t\rr)$ we can consider the corresponding Reidemeister-Milnor torsion $\tau(N,\phi)\in \Q(t)$ as defined in \cite{Mi62,Tu86,Tu01,FV10}. Before we continue we state the following proposition, which presumably is well-known, but for which we could not find a reference and whose proof is surprisingly fiddly.

\begin{proposition}\label{prop:surface-connected}
Let $N$ be a  3-manifold and let $\phi\in H^1(N;\Z)$ be a primitive class. If $\tau(N,\phi)\ne 0$, then any Thurston norm minimizing surface dual to $\phi$ is connected.
\end{proposition}

In the next theorem we will recall the well-known fact that $\tau(N,\phi)$ contains fiberedness information.
Before we state that theorem we introduce a few more definitions.
\bn
\item Let $p(t)=a_kt^k+a_{k+1}t^{k+1}+\dots+a_lt^l\in \zt$ be a Laurent polyonmial with $a_k\ne 0$ and $a_l\ne 0$.
\bn
\item we define the degree $\deg(p(t))=l-k$,
\item we say $p(t)$ is \emph{monic} if $a_k\in \{\pm 1\}$ and  $a_l \in \{\pm 1\}$,
\item we say $p(t)$ is \emph{top-monic} if  $a_l \in \{\pm 1\}$.
\en
\item For $f(t)=p(t)q(t)^{-1}\in \Q(t)$ with $p(t),q(t)\in \zt$ and $p(t),q(t)\ne 0$ we define the degree of $f(t)$ as  the difference of the degrees of $p(t)$ and $q(t)$. Furthermore we say $f(t)\in \Q(t)$ is \emph{monic} if it is the quotient of  two monic Laurent polynomials in $\zt$. 
\item  Given a properly embedded surface $\Sigma$ in a 3-manifold $N$
we often implicitly pick a thickening $\S\times [-1,1]\subset N$ which is orientation preserving. We then denote by  $\nu \S:=\S\times (-1,1)$ the corresponding open tubular neighborhood of $\S$. Furthermore we denote the two inclusion maps 
\[ \Sigma\,\to \,\Sigma\times \{\pm 1\}\,\subset\, N\sm \S\times (-1,1)\,=\,N\sm \nu \Sigma\]  by $\iota_{\pm}$.
\en

\begin{theorem}\label{thm:alexander-fibered}
Let $N$ be a 3-manifold and let  $\phi\in H^1(N;\Z)$ be a primitive class.
\bn
\item If $\phi$ is fibered, then $\tau(N,\phi)$ is monic and 
$ \deg(\tau(N,\phi))=\|\phi\|_T$. 
\item Conversely, if $\tau(N,\phi)$ is monic and 
$ \deg(\tau(N,\phi))=\|\phi\|_T$, then  any  Thurston norm minimizing surface $\S$ dual to $\phi$ is connected and  the two inclusion induced maps
\[ \iota_\pm\colon H_1(\S;\Z)\to  H_1(N\sm \nu \S;\Z)\]
are isomorphisms.
\en
\end{theorem}

This theorem has many parents; certainly Neuwirth \cite{New60,New65} and Rapaport~\cite{Rap60} in the 1960s were aware of it in the case of knots. For a proof in the context of general 3-manifolds we refer to \cite{FK06,Fr14} for the first part and Proposition~\ref{prop:surface-connected} together with \cite[Proposition~3.2]{FV11} for the second part.

 The second statement of the theorem says, that  
 if $\tau(N,\phi)$ is monic and  if furthermore
$ \deg(\tau(N,\phi))=\|\phi\|_T$, then  any  Thurston norm minimizing surface $\S$ `looks homologically like a fiber'. For that conclusion to hold one needs the monicness of $\tau(N,\phi)$ and the degree information on $\tau(N,\phi)$. 

In \cite{Co04,Ha05} Cochran and Harvey introduced non-commutative analogues of the usual Alexander polynomials and showed that their degrees give lower bounds on the Thurston norm and contained fiberedness information. 
In the subsequent discussion we will use the slight reformulation of these invariants given in \cite{Fr07}.

Let $N$ be a 3-manifold and let $\phi\in H^1(N;\Q)=\hom(\pi_1(N),\Q)$. We say that a homomorphism $\gamma\colon \pi_1(N)\to \G$ is \emph{admissible} if  $\gamma$ is an epimorphism such that $\phi$ factors through $\gamma$. By a slight abuse of notation we denote the corresponding unique homomorphism $\Gamma\to \Q$ by $\phi$ as well. We say that an admissible homomorphism
is \emph{tfea-admissible}, if $\Gamma$ is a  torsion-free elementary amenable group. The appeal of such homomorphisms lies in the fact that the group ring $\Z[\G]$ of a
torsion-free elementary amenable admits an  Ore localization  $\K(\G)$.

Given such a tfea-admissible homomorphism we can consider the corresponding torsion $\tau(N,\gamma)\in \K(\G)^\times_{\ab}\cup \{0\}$, where $\K(\G)^\times_{\ab}$ denotes the abelianization of the multiplicative group $\K(\G)^\times=\K(\G)\sm \{0\}$.
As we will see in Section~\ref{section:degrees}, the homomorphism $\phi\colon \G\to \Q$ gives rise to a degree function
\[ \deg_\phi\colon \K(\G)^\times_{\ab}\cup \{0\} \to \Q\cup \{-\infty\}.\]
The invariant $\deg_\phi(\tau(N,\gamma))$  basically  corresponds to the degrees of the noncommutative (or `higher-order') Alexander polynomials first introduced by Cochran \cite{Co04} and Harvey \cite{Ha05}. We refer to Section~\ref{section:relation-to-ch} for details.

As an example, if $\phi\in H^1(N;\Z)=\hom(\pi_1(N),\Z)=\hom(\pi_1(N),\ll t\rr)$, then $\gamma=\phi$  is admissible, and in this case the two definitions of the Reidemeister torsion agree and we have
$\deg_\phi(\tau(N,\phi))=\deg(\tau(N,\phi))$.

The following theorem is proved in \cite{Co04,Ha05}.

\begin{theorem}\label{thm:cochran-fibering-obstruction}
If $N$ is a 3-manifold and  $\phi\in H^1(N;\Z)$ is a primitive fibered class, then for any tfea-admissible homomorphism $\gamma$ we have
\[ \deg_\phi(\tau(N,\gamma))=\|\phi\|_T.\]
\end{theorem}

The theorem generalizes the aforementioned result that for a primitive fibered class $\phi$ the equality $\deg(\tau(N,\phi))=\|\phi\|_T$ holds. The degrees $ \deg_\phi(\tau(N,\gamma))$, with somewhat different notation, have been studied in great  detail by many authors \cite{CT08,Ha06,FK08a,FKK12,FH07,Fr07,FST15,LM06,LM08,Ho14,Tu02}.
So far the invariants $\deg_\phi(\tau(N,\gamma))\in \Z\cup \{-\infty\}$ are the only one's which have been extracted from the non-commutative setup. In particular there is no notion of `monicness' for these non-commutative invariants. 

Our goal in this paper is to introduce a natural companion to these noncommutative invariants which measures `monicness' in the sense that a generalization of 
Theorem~\ref{thm:alexander-fibered} to the non-commutative setup holds.

The idea hereby is to use Novikov-Sikorav homology as introduced by Novikov~\cite{No81} and~Sikorav \cite{Si87}.
Given a group $\G$ and a homomorphism $\phi\colon \G\to \Q$ Sikorav \cite{Si87} defined
\[ \zphig\,\, =\,\,\left\{ \ba{c}\mbox{all functions}\\f\colon \G\to \Z \ea \, \Big| \, \ba{c}\mbox{ for any $C\in \R$ there exist only finitely many}
\\\mbox{$g\in \G$ with $f(g)\ne 0$ and $\phi(g)<C$}\ea\right\}.\]
It is often helpful to  thinks of elements $\zphig$ as formal linear combinations of elements in $\G$. With this point of view  it is straightforward to verify that the `naive' multiplication on $\zphig$ makes sense and that it turns $\zphig$ into a ring. 
If $\Gamma$ is the infinite cyclic group generated by $t$ and if $\phi(t)=1$, then we make the canonical identification
\[ \Z_{\phi}[\ll t\rr\doubler\,\,=\,\,\Z[t^{-1},t\doubler\,\,=\,\,
\Big\{ \textstyle{\sum\limits_{i=k}^\infty} a_it^i\,|\, k\in \Z, a_i\in \Z\Big\}.\]
The ring on the right hand side is often referred to as the \emph{Novikov ring}.

Given a 3-manifold $N$, a primitive class $\phi\in H^1(N;\Z)=\hom(\pi_1(N),\Z)$ and an admissible homomorphism $\gamma\colon \pi_1(N)\to \Gamma$ we can consider the corresponding Novikov-Sikorav homology $H_1(N;\zphig)$. 

The following well-known theorem shows that the Novikov-Sikorav homology modules $H_1(N;\zphig)$ contain information about fiberedness.
For completeness' sake we provide a proof in Section~\ref{section:homology-zero-if-fibered}.

\begin{theorem} \label{thm:fib}
Let $N$ be a 3--manifold and let $\phi \in H^1(N;\Q)$. 
If $\phi$ is fibered, then for any admissible homomorphism $\gamma\colon \pi_1(N)\to \G$ we have
\[ H_1(N;\zphig)=0.\]
\end{theorem}

Thus we see that  Novikov-Sikorav homology gives fiberedness obstructions. 
The following lemma shows that in the special case $\gamma=\phi$ we recover some of  the fiberedness obstruction 
that is contained in the Reidemeister torsion $\tau(N,\phi)\in \Q(t)$.

\begin{lemma}\label{lem:novikov-zero}
Let $N$ be a 3-manifold  and let  
\[ \phi\in H^1(N;\Z)=\hom(\pi_1(N),\Z)=\hom(\pi_1(N),\ll t\rr)\] be a primitive class. Then $H_1(N;\ztt)=0$ if and only if $\tau(N,\phi)$ is monic.
\end{lemma}

Amazingly the next theorem gives a converse to Theorem~\ref{thm:fib}.
The  theorem is implicit in the work of Sikorav~\cite[p.~86]{Si87} and also Bieri~\cite[p.~953]{Bi07}. A closely related result in higher dimensions was proved by Ranicki~\cite[p.~622]{Ran95}. In Section~\ref{section:mainthmlow} we will provide a short self-contained proof of the theorem.

\begin{theorem}\label{thm:mainthmlow}
Let $N$ be a 3--manifold and let $\phi \in H^1(N;\Q)=\hom(\pi_1(N),\Q)$ be non-zero. Then $\phi$ is fibered if and only if
\[ H_1(N;\zphipi) =0.\]
\end{theorem}

As an aside we will show  in Section \ref{section:thurston}   how Theorem~\ref{thm:mainthmlow}
can be used to give an alternative proof of the following well--known result of Thurston's \cite{Th86}.

\begin{theorem} \label{thm:thurston}
Let $N$ be a 3-manifold and let $\phi\in H^1(N;\Q)$  be a fibered class. Then there exists an open neighborhood $U\subset H^1(N;\Q)$ of $\phi$, such that any $\psi\in U$ is also fibered.
\end{theorem}

The argument we give is close in spirit to other proofs in the literature, e.g.\ provided by Neumann~\cite{Nem79} and  Bieri--Neumann--Strebel  \cite[Theorems~A~and~E]{BNS87}.

Given a group $G$ we denote the commutator subgroup of  by $G^{(1)}$.
The following theorem is the main new result of the paper. 
For $\Gamma=\ll t\rr$ this theorem is a reformulation of 
Theorem~\ref{thm:alexander-fibered}.

\begin{theorem} \label{thm:iso-intro}
Let $N$ be a 3--manifold, let $\phi \in H^1(N;\Z)=\hom(\pi_1(N),\Z)$ be primitive and let $\gamma\colon \pi_1(N)\to \G$ be an admissible homomorphism. Assume that  $H_1(N;\Z_{\phi}[\G\doubler)=0$.
Then the following hold.
\bn
\item Every Thurston norm minimizing surface  dual to $\phi$ is connected.
\item For any  Thurston norm minimizing surface $\S$ dual to $\phi$
the inclusion induced maps
\[ \frac{\pi_1(\S)}{\ker(\gamma\colon \pi_1(\S)\to \G)}\xrightarrow{\,\,\i_\pm\,\,}
\frac{\pi_1(N\sm \nu \S)}{\ker(\gamma\colon \pi_1(N\sm \nu \S)\to \G)}  \]
are isomorphisms.
\item If $\G$ is a torsion-free elementary-amenable group, and if furthermore we have $\deg_\phi(\tau(N,\gamma))=\|\phi\|_T$,
then for any  Thurston norm minimizing surface $\S$ dual to $\phi$
the inclusion induced maps
\[ \frac{\pi_1(\S)}{\ker(\gamma\colon \pi_1(\S)\to \G)^{(1)}}\xrightarrow{\,\,\i_\pm\,\,}
\frac{\pi_1(N\sm \nu \S)}{\ker(\gamma\colon \pi_1(N\sm \nu \S)\to \G)^{(1)}}  \]
are isomorphisms.
\en
\end{theorem}

We get the strong conclusion of Statement (3) only by combining the information of the degree of the noncommutative torsion with the information on Novikov-Sikorav homology. Thus it seems to us that the vanishing of Novikov-Sikorav homology is the right generalization of an Alexander polynomial being monic.\\

As a final remark, in a future paper~\cite{FL16} the author and Wolfgang L\"uck  will also show that noncommutative Reidemeister torsions  always detect the Thurston norm of any irreducible 3-manifold that is not a closed graph manifold.

\subsection*{Conventions.} All groups are assumed to be finitely generated.
Given a ring $R$ we say that a matrix over $R$ is invertible if it has a right and a left inverse. We view elements in $R^n$ as column vectors.  All 3-manifolds are understood to be connected, orientable, compact with empty or toroidal boundary. In Section~\ref{section:connected} we will provide a proof of Proposition~\ref{prop:surface-connected}.
In Section~\ref{section:torsion} we will recall the definition of Reidemeister torsion over a skew-field and of the degree functions we are interested in. Finally in Section~\ref{section:opentwo} we will prove our main result, namely Theorem~\ref{thm:iso-intro}.

\subsection*{Organization.}
The paper is organized as follows. In Section~\ref{section:basics}
we recall the definition of Novikov--Sikorav homology and and we recall some of the basic properties. In Section~\ref{section:known}
we will outline proofs of statements about 3-manifolds, fiberedness and Novikov-Sikorav homology which explicitly or implicitly have already appeared in the literature. In particular we will provide proofs of 
Theorem~\ref{thm:mainthmlow} and Theorem~\ref{thm:thurston}.

\subsection*{Acknowledgments.} I would like to thank Andrew Ranicki for several very helpful conversations. I also would like to thank Jean--Claude Sikorav for providing me with a copy of his thesis and for sending me a draft version of his survey paper~\cite{Si15}. Finally I am very grateful to the referee for reading the paper thoroughly and for pointing out several mistakes in an earlier version of this paper. I am especially grateful that the referee pointed out that the original one-line proof of Proposition~\ref{prop:surface-connected} might be considered a little terse.  

This paper has had a very long genesis. It started out when I was a CRM-ISM postdoc at the Universit\'e du Qu\'ebec \`a Montr\'eal. I am grateful for the financial supported I received at that point. The recent research  was supported by the SFB 1085 `Higher Invariants' at the University of Regensburg, funded by the Deutsche Forschungsgemeinschaft (DFG).

\section{The definition of Novikov--Sikorav homology and basic properties}\label{section:basics}
In this section we will introduce Novikov--Sikorav homology and we will prove several basic facts about Novikov--Sikorav homology that we will use later on.

\subsection{The definition of Novikov--Sikorav homology}
Given a connected CW-complex $X$ and a subcomplex $Y\subset X$ we denote by $p\colon \wti{X}\to X$ its universal covering and we write $\wti{Y}=p^{-1}(Y)$.
The fundamental group $\pi=\pi_1(X)$ acts naturally on the left of $C_*(\wti{X},\wti{Y})$, thus we can view $C_*(\wti{X},\wti{Y})$ as a chain complex of left $\Z[\pi]$-modules.  Given group homomorphisms $\gamma\colon \pi\to \G$ and $\phi\colon \Gamma\to \Q$ we can view $\zphig$ as a $\Z[\pi]$-left module using left multiplication. We define 
\[ C^*(X,Y;\zphig):=\hom_{\Z[\pi]}(C_*(\wti{X},\wti{Y}),\zphig)\]
and we denote the corresponding cohomology modules by $H^*(X,Y;\zphig)$.
We can use the canonical involution on the group ring $\Z[\pi]$ to turn   $C_*(\wti{X},\wti{Y})$ into a chain complex of right $\Z[\pi]$-modules. We then define
\[ C_*(X,Y;\zphig):=C_*(\wti{X},\wti{Y})\otimes_{\Z[\pi]}\zphig\]
and we denote the corresponding homology modules by $H_*(X,Y;\zphig)$.
These modules are often referred to as Novikov--Sikorav homology of $(X,Y)$. As usual we will drop $Y$ from the notation if $Y=\emptyset$.

\subsection{Stably finite rings}
In the following we say that a ring $R$ is \emph{stably finite}
if any epimorphism $R^n\to R^n$ of right $R$-modules is in fact an isomorphism. We refer to \cite[p.~5]{La99} for more information and background on stably finiteness.
The following proposition says that group rings and their Novikov--Sikorav completions are stably finite.

\begin{proposition}\label{prop:stably-finite}
For any group $\G$ the group ring $\Z[\G]$ is stably finite. Furthermore, for any $\phi\in \hom(\G,\Q)$ the Novikov--Sikorav completion $\zphig$ is stably finite. 
\end{proposition}

The first statement is proved by Kaplansky~\cite[p.~122]{Ka69} \cite[Section~4.3]{Roe03} (note that Kaplansky uses an alternative formulation of stable finiteness given e.g.\ in \cite[p.~5]{La99}). The statement for Novikov--Sikorav completions is precisely~\cite[Theorem~1]{Ko06}. We also refer to \cite{Si15} for an alternative proof.

\begin{remark}
If $R$ is stably finite, then (see~\cite[p.~5]{La99}) any matrix that has a left (respectively right) inverse also has a right (respectively left) inverse. In particular for group rings and Novikov--Sikorav completions we do not have to worry about the different notions of invertibility of matrices. 
\end{remark}

We can now formulate the following lemma.

\begin{lemma} \label{lem:turaev}
Let $R$ be a ring and let
\[ \mathcal{C}\,\,=\,\,0\to R\xrightarrow{C} R^n \xrightarrow{B} R^n \xrightarrow{A} R\to 0\]
be a complex of $R$ right-modules. We write
\[ \ba{rcl}
A&=&\bp a_1&a_2&\dots&a_n\ep,\\
B&=& \bp *&*\\
*&B'\ep, \\
C&=&\bp c_1&c_2&\dots&c_n\ep^t,
\ea \]
where $B'$ is an $(n-1)\times (n-1)$--matrix over $R$.
Assume that $c_1$ and $a_1$ are units in $R$. Then the following are equivalent:
\bn
\item $B'$ is invertible over $R$,
\item  the complex is acyclic,
\en
and if $R$ is stably finite, then the above are equivalent to
\bn
\item[(3)] $H_1(\mathcal{C}_*)=0$.
\en
\end{lemma}

\begin{proof}
We will make use of the following  two invertible matrices
\[
P=\bp 1&0&0&\dots&0 \\ c_2c_1^{-1}&1&0&\dots&0\\ c_3c_1^{-1}&0&1&\dots&0\\ \vdots&\vdots&&\ddots&\vdots\\ c_nc_1^{-1}&0&0&\dots&1\ep
\mbox{ and }
Q=\bp 1&-a_1^{-1}a_2&-a_1^{-1}a_3&\dots&-a_1^{-1}a_n\\ 0&1&0&\dots&0 \\ 0&0&1&\dots&0\\ \vdots&\vdots&&\ddots&\vdots\\ 0&0&0&\dots&1\ep
\]
Consider the following commutative diagram of complexes
\[  \xymatrix@C1.3cm@R0.8cm{
0\ar[r]& R \ar[r]^{C}& R^n \ar[r]^{B}& R^n \ar[r]^{A}& R\ar[r]& 0\\
0\ar[r]& R\ar[r]_-{P^{-1}C}\ar[u]^{\id}& R^n \ar[r]_{Q^{-1}BP}\ar[u]^P& R^n \ar[r]_{AQ}\ar[u]^Q& R\ar[r]\ar[u]^{\id}& 0.} \]
Note that
\[ \ba{rcl} P^{-1}C&=&\bp c_1&0&\dots&0\ep^t,\\
Q^{-1}BP&=& \bp *&*\\
*&B'\ep, \\
AQ&=&\bp a_1&0&\dots&0\ep.\ea \]
It follows from an elementary linear algebra argument, see e.g.\ \cite[Theorem~2.2]{Tu01}, that the bottom sequence is acyclic if and only if the map represented by $B'$ is a bijection, i.e. if $B'$ is invertible.

Now suppose that $R$ is stably finite. Then $B'$ is invertible if and only if the corresponding homomorphism represented by $B'$  is surjective. But it is clear that this is the case if and only if $H_1(\mathcal{C}_*)=0$. 
\end{proof}

\subsection{Invertible matrices over  Novikov--Skorav rings}

We adopt the following notations. Given a homomorphism $\phi\colon \G\to \Q$ we write $\Gamma_\phi=\ker\{\phi\colon \G\to \Q\}$. Furthermore given $C\in \Q$ we write
\[ \Z_\phi^{>C}[\G\doubler =\big\{ \gamma\colon \G\to \Z \, \big|\,
\g\in \Z_\phi[\G\doubler\mbox{ and } \gamma(g)=0\mbox{ for all }g \mbox{ with }\phi(g)\leq C\big\}.\]
Clearly if $p\in \Z_\phi^{>C}[\G\doubler $ and $q\in \Z_\phi^{>D}[\G\doubler $, then $p\cdot q\in \Z_\phi^{>C+D}[\G\doubler$.

Note that given a non--zero matrix $A$ over $\zphig$ there exists a unique $C$ such that $A=A'g+A''$ where
$\phi(g)=C$, $A'$ is non--zero and defined over $\Z[\G_{\phi}]$, and $A''$ is a matrix over $\Z_\phi^{>C}[\G\doubler $.

The following well--known lemma is the cornerstone of calculations over $\zphig$. Here we say that a square matrix $A$ is \emph{non--degenerate} if the zero matrix is the only matrix $B$ such that $AB=0$.

\begin{lemma} \label{lem:inv}
Let $\G$ be a group and let $\phi\colon \G\to \Q$ be a homomorphism. 
Let $A$ be a non--zero square matrix over $\zphig$. We write $A=A'g+A''$ where
$\phi(g)=C$, $A'$ is a non--zero matrix defined over $\Z[\G_{\phi}]$ and $A''$ is a matrix over $\Z_\phi^{>C}[\G\doubler $.
Assume that $A'$ is non--degenerate. Then
$A$ is invertible over $\zphig$ if and only if $A'$ is invertible over $\Z[\G_{\phi}]$.
\end{lemma}

\begin{remark}
Consider the matrix
\[ A=\bp 1&0\\ 0&t \ep=\underset{=A'}{\underbrace{\bp 1&0 \\ 0&0\ep}} +\underset{=A''}{\underbrace{t\bp 0&0\\ 0&1\ep}}.\]
It is invertible over $\Z[t^{-1},t\doubler$, but $A'$ is not invertible. This shows that in the above lemma it is necessary to assume that $A'$ is non--degenerate.
\end{remark}

\begin{proof}
First assume that $A$ has an inverse $B$ over $\zphig$.
Then there exists a unique $D\in \Q$ such that  $B=B'h+B''$ where
$ \phi(h)=D$, $B'\ne 0$ is defined over $\Z[\G_{\phi}]$,  and $B''$ is a matrix over $\Z_\phi^{>D}[\G\doubler $.
We compute
\[ \id=A\cdot B=(A'g+A'')\cdot (B'h+B'')=A'gB'g^{-1}\cdot gh+R\]
where $R$ is a matrix over $\Z_{\phi}^{>C+D}[\G\doubler$. Note that $gB'g^{-1}$ is a matrix over $\Z[\G_\phi]$.
Since $A'$ is by assumption non--degenerate it follows that $A'gB'g^{-1}=\id$ and $R=0$, i.e.
 $A'$ is invertible over $\Z[\G_{\phi}]$.

On the other hand assume we are given  $A=A'g+A''$ with $A'$ invertible over $\Z[\G_\phi]$ and $A''$ a matrix in $\Z_\phi^{>\phi(g)}[\G\doubler$.
After multiplication by $g^{-1}A'^{-1}$ we can assume that $A'=\id$, $g=e$ and $C=0$. So now we have to show that $\id+A''$ with $A''$ defined over
$\Z_{\phi}^{>0}[\G\doubler$ is invertible. It is easy to see that the fact that  $A''$ lies in $\Z_{\phi}^{>0}[\G\doubler $ implies that the infinite sum
\[ \id+\sum_{k=1}^\infty (-1)^k (A'')^k, \]
defines a matrix over $\zphig$. Clearly this matrix is  an inverse to $\id+A''$.
\end{proof}

For future reference we also record the following well-known special case of Lemma~\ref{lem:inv}.

\begin{lemma} \label{lem:inv-polynomial}
Let $p(t)\in \zt$ be a polynomial. Then $p(t)$ is invertible over $\ztt$ if and only if $p(t)$ is top-monic.
\end{lemma}

\subsection{Basic properties of Novikov--Sikorav homology}

The following well-known lemma gives us convenient chain complexes for computing (twisted) homology groups of 3-manifolds.

\begin{lemma}\label{lem:chain-complex-n}
Let $N$ be a 3-manifold and let $\phi\in H^1(N;\Q)$ be non-trivial.
\bn
\item 
If $N$ is closed, then $C_*(\wti{N})$, viewed as a chain complex of $\Z[\pi]$ right modules, is homotopy equivalence to a chain complex of the form 
\[ 0\to \Z[\pi]\xrightarrow{C} \Z[\pi]^n \xrightarrow{B} \Z[\pi]^n \xrightarrow{A} \Z[\pi]\to 0\]
where 
\[ \ba{rcl}
A&=&\bp 1-g&*&\dots&*\ep,\\
B&=& \bp *&*\\
*&B'\ep, \\
C&=&\bp 1-h&*&\dots&*\ep^t,
\ea \]
where $\phi(g)\ne 0$ and $\phi(h)\ne 0$ and where $B'$ is an $(n-1)\times (n-1)$--matrix.
\item 
If $N$ has non-empty boundary, then $C_*(\wti{N})$, viewed as a chain complex of $\Z[\pi]$ right modules, is homotopy equivalence to a chain complex of the form 
\[ 0\to \Z[\pi]^{n-1} \xrightarrow{B} \Z[\pi]^n \xrightarrow{A} \Z[\pi]\to 0\]
where 
\[ \ba{rcl}
A&=&\bp 1-g&*&\dots&*\ep,\\
B&=& \bp *\\
B'\ep,
\ea, \]
where $\phi(g)\ne 0$ and where $B'$ is an $(n-1)\times (n-1)$--matrix.
\en
\end{lemma}

\begin{proof}
If $N$ is closed, then the lemma follows easily from picking a CW-structure with one 0-cell and one 3-cell. If $N$ has non-empty boundary, then we retract $N$ onto a 2-dimensional CW-complex with one 0-cell.  We refer to \cite[Section~5]{McM02} or  \cite[Proof~of~Lemma~6.2]{FK08b} for details. 
\end{proof}

The following lemma can be viewed as a non--commutative version of
\cite[Theorem~5.5]{Pa06}.

\begin{lemma} \label{lem:h0123}
Let $N$ be a 3--manifold, let $\phi \in H^1(N;\Q)=\hom(\pi_1(N),\Q)$ be non-zero and let $\gamma\colon \pi_1(N)\to \G$ be an admissible  homomorphism. If $H_1(N;\zphig)=0$, then $H_i(N;\zphig)=0$ for all $i$.
\end{lemma}

\begin{proof}
First we consider the case that $N$ is closed. We consider the chain complex of Lemma~\ref{lem:chain-complex-n} (1). We obtain a chain complex that computes $H_*(N;\zphig)$ by tensoring with $\zphig$. In particular the boundary matrices over $\zphig$ are given by applying $\gamma$ to the  matrices. It  follows easily from Lemma~\ref{lem:inv} that $1-\gamma(g)$ and $1-\gamma(h)$ are invertible over $\zphig$. By Proposition~\ref{prop:stably-finite} the ring $\zphig$ is stably finite. The desired statement now follows from Lemma~\ref{lem:turaev}.

The case that $N$ has non-empty boundary is proved essentially the same way. Now one applies Lemma~\ref{lem:chain-complex-n} (2) and an obvious analogue to Lemma~\ref{lem:turaev} for chain complexes of length 2. In the interest of space we leave the details to the reader.
\end{proof}

\begin{lemma} \label{lem:delta-t-nonzero}
Let $N$ be a 3--manifold and
 let 
\[ \phi \in H^1(N;\Z)=\hom(\pi_1(N),\Z)=\hom(\pi_1(N),\ll t\rr)\] be non-zero. Furthermore let $\gamma\colon \pi_1(N)\to \G$ be an admissible  homomorphism. If $H_1(N;\zphig)=0$, then  $H_1(N;\Z[t^{-1},t\doubler)=0$.
\end{lemma}

\begin{proof}
We prove the lemma in the closed case. The case of non-trivial boundary is once again left to the reader.

First we consider the case that $N$ is closed. We consider the chain complex of Lemma~\ref{lem:chain-complex-n} (1). We obtain a chain complex that computes $H_*(N;\zphig)$ by applying $\gamma$  to the matrices. Since $\phi$ is admissible it  follows easily from Lemma~\ref{lem:inv} that $1-\gamma(g)$ and $1-\gamma(h)$ are invertible in $\zphig$. Similarly  we obtain a chain complex that computes $H_*(N;\Z[t^{-1},t\doubler)$ by applying $\phi\colon \pi_1(N)\to \ll t\rr$ to the matrices. It follows again  from Lemma~\ref{lem:inv} that $1-\phi(g)$ and $1-\phi(h)$ are invertible in $\Z[t^{-1},t\doubler$.

Since $\zphig$ and $\Z[t^{-1},t\doubler$ are stably finite we can appeal to  Lemma~\ref{lem:turaev}. Our assumption that $H_1(N;\zphig)=0$ implies that $\g(B')$ is invertible over $\zphig$. The homomorphism $\phi\colon \G\to \ll t\rr$ induces a ring homomorphism $\zphig\to \Z[t^{-1},t\doubler$. Thus we see that $\phi(B')$ is also invertible over $\Z[t^{-1},t\doubler$. But, once again appealing to Lemma~\ref{lem:turaev}, this implies that $H_1(N;\Z[t^{-1},t\doubler)=0$.
\end{proof}

\begin{lemma} \label{lem:nonzero-on-boundary}
Let $N$ be a 3--manifold with non-empty boundary, furthermore let $\phi \in H^1(N;\Z)=\hom(\pi_1(N),\Z)$ be non-zero and let $\gamma\colon \pi_1(N)\to \G$ be an admissible  homomorphism. If $H_1(N;\zphig)=0$, then the restriction of $\phi$ to any boundary component is non-zero and  $H_*(\partial N;\zphig)=0$.
\end{lemma}

\begin{proof}
By Lemma~\ref{lem:novikov-zero} we know that $\tau(N,\phi)$ is in particular non-zero. By standard arguments, see e.g.\ \cite[Chapter~4]{Tu01}, this implies that $b_1(\ker(\phi\colon \pi_1(N)\to \Z))$ is finite. By another standard argument, see e.g.\ \cite[Section~6]{McM02}, this implies that the restriction of $\phi$ to any boundary component of $N$ is non-zero. Then it follows easily from a straightforward calculation that  $H_*(\partial N;\zphig)=0$.
\end{proof}

We also have the following lemma which is basically a consequence of Poincar\'e duality.

\begin{lemma} \label{lem:leftright}
Let $N$ be a 3--manifold, let $\phi \in H^1(N;\Q)=\hom(\pi_1(N),\Q)$  and let $\gamma\colon \pi_1(N)\to \G$ be an admissible homomorphism. Then
$H_1(N;\zphig)=0$ if and only if $H_1(N;\Z_{-\phi}[\Gamma\doubler)=0$.
\end{lemma}

\begin{proof}
Suppose that $H_1(N;\zphig)=0$. It follows from Lemma~\ref{lem:h0123}, from Lemma~\ref{lem:nonzero-on-boundary} and from the long exact sequence in homology that $H_*(N,\partial N;\zphig)=0$. 
It now follows from Poincar\'e Duality and  the  Universal Coefficient Spectral Sequence (cf.\ \cite[p.~515]{McC01} and \cite{Le77}) that $H_1(N;\Z_{-\phi}[\Gamma\doubler)=0$.
\end{proof}

Now we are also in a position to prove Lemma~\ref{lem:novikov-zero}.
For the reader's convenience we recall the statement of the lemma.
\\

\noindent \textbf{Lemma~\ref{lem:novikov-zero}.}
\emph{Let $N$ be a 3-manifold and let  
$\phi\in H^1(N;\Z)=\hom(\pi_1(N),\ll t\rr)$ be a primitive class. Then $H_1(N;\ztt)=0$ if and only if $\tau(N,\phi)$ is monic.}
\\

\begin{proof}
We prove the lemma in the closed case. The case of non-trivial boundary is  left to the reader.  By Lemma~\ref{lem:chain-complex-n} (1) the chain complex
$C_*(\wti{N})\otimes_{\Z[\pi_1(N)]}\zt$ is of the form
\[ 0\to \zt\xrightarrow{C} \zt^n \xrightarrow{B} \zt^n \xrightarrow{A} R\to 0\]
where
\[ \ba{rcl}
A&=&\bp 1-t^k&*&\dots&*\ep,\\
B&=& \bp *&*\\
*&B'\ep, \\
C&=&\bp 1-t^l&*&\dots&*\ep^t,
\ea \]
where $B'$ is an $(n-1)\times (n-1)$--matrix over $\zt$ and where $k\ne 0$ and $l\ne 0$.

By \cite[Theorem~2.2]{Tu01} (see also Lemma~\ref{lem:torsion-2-complex}) we have \[\tau(N,\phi)=\det(B')\cdot (1-t^k)^{-1}\cdot (1-t^l)^{-1}.\]
Thus we see that $\tau(N,\phi)$ is monic if and only if $\det(B')$ is monic. 
But by Lemma~ \ref{lem:inv-polynomial} this is equivalent to $\det(B')$ being invertible over $\ztt$ and over $\zttleft$, which by Lemma~\ref{lem:turaev} is equivalent to $H_*(N;\ztt)=H_*(N;\zttleft)=0$, which  by Lemma~\ref{lem:h0123}  is equivalent to $H_1(N;\ztt)=H_1(N;\ztt)=0$. Finally by 
Lemma~\ref{lem:leftright} this is equivalent to $H_1(N;\ztt)=0$.
\end{proof}

\section{Proofs of known results}\label{section:known}

In this section we will outline proofs of statements about 3-manifolds, fiberedness and Novikov-Sikorav homology which explicitly or implicitly have already appeared in the literature. 
We do not make any claims to originality of the results.  We give  short self-contained proofs, some of which  we found independently of the much earlier proofs. We hope that this exposition of the beautiful subject of Novikov-Sikorav homology will be of interest to some readers.

%
\subsection{Proof of Theorem~\ref{thm:fib}} 
\label{section:homology-zero-if-fibered}
\label{section:fib}

We will prove Theorem~\ref{thm:fib} in the closed case.
We refer to \cite[Sections~6.1~and~6.2]{Fr07} for the standard technique for adapting  the proof to the case that $N$ has non-empty boundary. We leave the  details to the reader. 

Let $N$ be a 3--manifold,  let   $\phi \in H^1(N;\Q)=\hom(\pi_1(N),\Q)$ be a fibered class and let $\gamma\colon \pi_1(N)\to \G$ be an admissible homomorphism.
Note that for  $n\ne 0$ the class $\phi$ is fibered  if and only if $n\phi$ is fibered, furthermore we
have  $\Z_{\phi}[\pi\doubler=\Z_{n\phi}[\pi\doubler$. Therefore it follows that
we can assume that $\phi$ is in fact an  integral primitive class in $H^1(N;\Z)$.

We can view $N$ as the mapping torus of a surface diffeomorphism $\psi\colon \S\to \S$ of a surface $\S$ of genus $g$.
We pick a CW--structure for $\S$ with one 0--cell $v$, $2g$ 1--cells $e_1,\dots,e_{2g}$ and one 2--cell $f$.
Now let $\psi'$ be a cellular approximation of $\psi$.
Then $N$ is homotopy equivalent to the mapping torus $N'$ of $\psi'$.

It suffices to  show that $H_1(N';\zphig)=0$.
Note that every $i$--cell $c$ of $\S=\Sigma\times  \{-1\}$ gives rise to an $i$--cell of $N$ (which we also denote by $c$) and gives rise to a product 
$i+1$--cell of $N'$, denoted by $c'$. So we get a CW--structure of $N'$ with
\bn
\item one cell of dimension 0: $v$,
\item $2g+1$ cells of dimension 1: $v',e_1,\dots,e_{2g}$,
\item $2g+1$ cells of dimension 2: $e_1',\dots,e_{2g}',f$ and
\item one cell of dimension 3: $f'$.
\en
Note that $e_1,\dots,e_{2g}$ give rise to elements in $\pi_1(N')$ which we denote by the same symbols. Furthermore $v'$ gives rise to an element in $\pi_1(N')$ which we denote by $t$. With this notation the boundary of the 2--cell $e_i'$ is represented by $e_i^{-1}t\psi'(e_i)t^{-1}$.
Note that $\gamma(e_1),\dots,\gamma(e_{2g})\in \Gamma_\phi=\ker\{\phi\colon \G\to \Z\}$ and that  $\phi(t)=1$.

Picking appropriate lifts of these (oriented) cells of $N'$ to cells of the universal cover
$\wti{N}'$ we get bases for the
$\Z[\pi_1(N)]$--modules $C_i(\wti{N}')$, such that if $A_i$ denotes the matrix corresponding to
$\partial_i$, then the $A_i$ are of the form
\[ \ba{rcl}
A_1&=&(-1+t,*,\dots,*),\\[2mm]
A_2&=& \bp * &*\\ -\id +t\Big(\tmfrac{\partial \psi'(e_i)}{\partial e_j}\Big)&* \ep   \\[2mm]
A_3&=&(*,\dots,*,-1+t)^t.
   \ea \]
with  unspecified entries denoted by $*$. Here $A_2$ is a $(2g+1)\times (2g+1)$--matrix and $-\id +t\big(\frac{\partial \psi'(e_i)}{\partial e_j}\big)$ is a
$2g\times 2g$--submatrix. Note that $\gamma\big(\frac{\partial \psi'(e_i)}{\partial e_j}\big)$ is defined over $\Z[\G_{\phi}]$.

The boundary maps of the $\zphig$--complex $C_*(\wti{N}')\otimes_{\Z[\pi_1(N')]}\zphig$ are represented by the matrices
$\gamma(A_i)$. Since $\gamma$ is admissible it follows from Lemma~\ref{lem:inv} that the submatrices $\gamma(-1+t)$ and $\gamma\big(-\id +t\big(\frac{\partial \psi'(e_i)}{\partial e_j}\big)\big)$ are invertible over $\zphig$. It follows from Lemma~\ref{lem:h0123} that   $H_*(N';\zphig)=0$.
This completes the proof of Theorem~\ref{thm:fib}.

\subsection{Proof of Theorem~\ref{thm:mainthmlow}}\label{section:mainthmlow}
In this section we will prove a purely group theoretic statement that implies  Theorem~\ref{thm:mainthmlow}.

First we recall  the notion of a HNN-extension of a group $B$. Let $A\subset B$ be a subgroup, let $\gamma\colon A\to B$ be a monomorphism and let $B=\langle X|R\rangle$ be a presentation for $B$.
Then the corresponding HNN--extension is given by the presentation
\[  \langle X,t | R, \{ tat^{-1}=\gamma(a) |a\in A\} \rangle.\]
We will write $\langle B,t | t^{-1}at=\gamma(a) \rangle$ for such an extension. We say that the HNN--extension is \emph{ascending} if $A=B$.

It is well--known, see e.g.\ \cite[Theorem~B*]{Str84}, that any pair $(\pi,\phi)$ with $\pi$ a group and with $\phi\in \hom(\pi,\Q)$ non-zero can be presented by an HNN-extension. This means that
 there exists an HNN--extension $\langle B,t | t^{-1}at=\gamma(a) \rangle$ with $A$ and $B$ finitely generated groups and an isomorphism
$\pi\to \langle B,t | tat^{-1}=\gamma(a) \rangle$ such that the following diagram commutes
\[ \xymatrix{ \pi\ar[dr]_\phi \ar[rr]^-\cong&& \langle B,t | tat^{-1}=\gamma(a) \rangle \ar[dl] \\
&\Q&}\]
where the diagonal map on the right sends $t$ to a \emph{positive} rational number and it sends all elements in $B$ to $0$. By \cite[Proposition~4.3]{BNS87} the Bieri-Neumann-Strebel \cite{BNS87} invariant $\Sigma(\pi)\subset \hom(\pi,\Q)=H^1(N;\Q)$ of $\pi$ is precisely the set of all non-zero $\phi$'s in $\hom(\pi,\Q)=H^1(N;\Q)$ such that $(\pi,\phi)$ can be represented by an ascending HNN-extension. 

The following theorem follows from \cite[Theorem~E]{BNS87}, which in turn builds on Stallings' fibering theorem \cite{St61}, and the resolution of the Poincar\'e Conjecture by Perelman.

\begin{theorem}\label{thm:3-manifold-bns}
Let $N$ be a 3-manifold and let $\phi\in H^1(N;\Q)=\hom(\pi_1(N),\Q)$.
Then $\phi$ is fibered if and only if $\phi\in \Sigma(\pi_1(N))$. 
\end{theorem}

In light of Theorems~\ref{thm:fib} and~\ref{thm:3-manifold-bns} the desired Theorem~\ref{thm:mainthmlow} is now a consequence of the following purely group theoretic theorem.

\begin{theorem}\label{mainthm-group-theory}
Let $\pi$ be a finitely presented group and let $\phi \in H^1(\pi;\Q)=\hom(\pi,\Q)$ be non-zero. If 
$H_1(\pi;\zphipi) =0$, then $\phi\in \Sigma(\pi)$.
\end{theorem}

As we mentioned before, this theorem can already be found, with somewhat different language, in the work of Sikorav~\cite{Si87} and Bieri~\cite{Bi07}.

\begin{proof} 
Let $\pi$ be a finitely presented group and let $\phi \in \hom(\pi,\Q)$ be non-zero. As in the proof of Theorem~\ref{thm:fib} we can assume that $\phi\colon \pi\to \Z$ is an epimorphism. We write $\pi_\phi:=\ker\{\phi\colon \pi\to \Z\}$.

We pick an identification
\[ \pi=\langle B,t | tat^{-1}=\gamma(a)\mbox{ for all $a\in A\subset B$} \rangle\]
with $A$ and $B$ finitely generated as above. Now assume that
 $H_*(\pi;\zphipiright)=0$. We need to show that the inclusion map $\i\colon A\to B$ is an isomorphism, i.e.\ we still need to show it is surjective.

We first study the relationship between  $H_*(\pi;\zphipiright)$ and $\i$.
Note that $\phi$ vanishes on $A$ and $B$. Furthermore note that $\Z_{\phi}[\pi\doubler$ is flat over $\Z[\pi_\phi]$ since
 $\Z_{\phi}[\pi\doubler$ is a direct limit of free $\Z[\pi_\phi]$--modules. Therefore we get canonical isomorphisms
 \[ \ba{rclcl} H_i(A;\Z_{\phi}[\pi\doubler) &=&H_i(A;\Z[\pi_\phi])\otimes_{\Z[\pi_\phi]}\Z_{\phi}[\pi\doubler &=& H_i(A;\Z[\pi_\phi])\otimes_{\Z}\ztt,\\
 H_i(B;\Z_{\phi}[\pi\doubler) &=&H_i(B;\Z[\pi_\phi])\otimes_{\Z[\pi_\phi]}\Z_{\phi}[\pi\doubler&=&H_i(B;\Z[\pi_\phi])\otimes_{\Z}\ztt.\ea \]
Thus the Meyer--Vietoris type sequence can be written as follows (we refer to \cite{Bi75} and also \cite{FK06} for details)
\[
\ba{ccccllll}
&\hspace{-0.1cm}\hspace{-0.11cm}&&\dots&\hspace{-0.1cm}\to\hspace{-0.1cm}&\hspace{-0.1cm}
H_2(\pi;\Z_{\phi}[\pi\doubler)\to \\
\hspace{-0.1cm}\to\hspace{-0.1cm}\hspace{-0.1cm}&\hspace{-0.1cm}H_1(A;\Z[\pi_\phi])\hspace{-0.11cm}\otimes_\Z\hspace{-0.11cm}\ztt\hspace{-0.1cm}&\hspace{-0.1cm}
\xrightarrow{\i-t\gamma}\hspace{-0.1cm}&\hspace{-0.1cm}H_1(B;\Z[\pi_\phi])\hspace{-0.11cm}\otimes_\Z\hspace{-0.11cm}\ztt\hspace{-0.1cm}&\hspace{-0.1cm}\hspace{-0.1cm}
\to\hspace{-0.1cm}\hspace{-0.1cm}&\hspace{-0.1cm}
H_1(\pi;\Z_{\phi}[\pi\doubler)\to \\
\hspace{-0.1cm}\to\hspace{-0.1cm}\hspace{-0.1cm}&\hspace{-0.1cm}H_0(A;\Z[\pi_\phi])\hspace{-0.11cm}\otimes_\Z\hspace{-0.11cm}\ztt\hspace{-0.1cm}&\hspace{-0.1cm}
\xrightarrow{\i-t\gamma}\hspace{-0.1cm}&\hspace{-0.1cm}H_0(B;\Z[\pi_\phi])\hspace{-0.11cm}\otimes_\Z\hspace{-0.11cm}\ztt\hspace{-0.1cm}&\hspace{-0.1cm}\hspace{-0.1cm}
\to\hspace{-0.1cm}\hspace{-0.1cm}&\hspace{-0.1cm}
0.
\ea \]
Now assume that $\i\colon A\to B$ is not a surjective.
We will obtain a contradiction from the above long exact sequence by showing that this assumption implies  that
\[
H_0(A;\Z[\pi_\phi])\otimes_{\Z}\ztt \xrightarrow{\i-t\gamma} H_0(B;\Z[\pi_\phi])\otimes_{\Z}\ztt \] is not injective.
By the standard calculation of zeroth twisted homology, see e.g.\ \cite[Chapter~VI.3]{HS97}, we have a commutative diagram
\[ \xymatrix{
H_0(A;\Z[\pi_\phi])\ar[rr]^{\i}\ar[d]^\cong &&H_0(B;\Z[\pi_\phi])\ar[d]^\cong\\
\Z[\pi_\phi/A]\ar[rr]^{\i}&&\Z[\pi_\phi/B]. }\]
Since we assume that $\i\colon A\to B$ is not surjective it follows  from  the above commuting diagram that the map
\[
\i\colon  H_0(A;\Z[\pi_\phi]) \to H_0(B;\Z[\pi_\phi])\]
is surjective but not injective.
Therefore we can find a non-zero $f_0\in H_0(A;\Z[\pi_\phi])$ such that $\i(f_0)=0$ and we can  iteratively find a sequence of elements  $f_i\in H_0(A;\Z[\pi_\phi])$ with $i\geq 0$ such that
\[ \i(f_{i+1})\,=\,\gamma(f_i)\mbox{ for all $i\in \N$.} \]
Clearly $\sum_{i=0}^\infty f_it^i$  defines a non--trivial element in the kernel of the map
\[ H_0(A;\Z[\pi_\phi])\otimes_{\Z}\ztt \xrightarrow{\i-t\gamma} H_0(B;\Z[\pi_\phi])\otimes_{\Z}\ztt.\]
As observed above, this implies that $H_1(\pi;\zphipiright)\ne 0$. Thus we have obtained a contradiction.
\end{proof}
%
\subsection{Proof of Theorem~\ref{thm:thurston}}\label{section:thurston}
For the reader's convenience we recall the statement of Theorem~\ref{thm:thurston}.\\

\noindent \textbf{Theorem~\ref{thm:thurston}.}
\emph{Let $N$ be a fibered 3-manifold and let $\phi\in H^1(N;\Q)$ be a fibered class. Then there exists an open neighborhood $U\subset H^1(N;\Q)$ of $\phi$, such that any
$\psi\in U$ is fibered. }\\

The proof of Theorem~\ref{thm:thurston} will require the remainder of the section. So let $N$ be a 3-manifold and let $\phi\in H^1(N;\Q)$ be a fibered class.
As in the proof of Theorem~\ref{thm:fib} we can assume that $\phi\in H^1(N;\Q)$ is in fact a primitive element in $H^1(N;\Z)$.

Now assume that $\phi$ is a fibered class with corresponding fiber $\S$ and monodromy $\psi\colon \S\to \S$. We let $\psi'\colon \S\to \S$ be the map and $N'$ be the CW--complex from Section \ref{section:fib}.
 In the following we write $\pi=\pi_1(N)=\pi_1(N')$.
By Theorem~\ref{thm:mainthmlow} it is enough to show that there exists an open neighborhood $U$ of $\phi$ such that
$H_1(N';\Z_{\psi}[\pi\doubler)=0$ for all $\psi\in U$.

We let $t\in \pi$ be the element with  $\phi(t)=1$ as in Section \ref{section:fib}.
As in Section \ref{section:fib} we see that we can find bases for the complex $C_*(\wti{N}')$ of $\Z[\pi]$--right modules such that the corresponding
boundary maps $\partial_i$ are represented by matrices $A_i$  of the form
\[ \ba{rcl}
A_1&=&(-1+t,*,\dots,*),\\
A_2&=& \bp * &*\\ -\id +t\Big(\tmfrac{\partial \psi'(e_i)}{\partial e_j}\Big)&* \ep,   \\
A_3&=&(*,\dots,*,-1+t)^t,
   \ea \]
where $e_1,\dots,e_n$ are generators of $\pi_1(\Sigma)$, in particular they lie in $\pi_\phi=\ker\{\phi\colon \pi\to \Z\}$.
Now let $V\subset H^1(N;\Q)$ be an open neighborhood of $\phi$ such that $\psi(t)\ne 0$ for all $\psi\in V$.
Note that this implies that $-1+t$ is invertible over $\Z_{\psi}[\pi\doubler$ for any $\psi\in V$.

We write 
\[ B_2=-\id+t\Big(\smfrac{\partial \psi'(e_i)}{\partial e_j}\Big)_{i,j=1,\dots,n}.\]
As in Section \ref{section:fib} we get that the matrix $B_2$ over $\Z[\pi']$ is invertible in $\Z_{\phi}[\pi\doubler$.
Now we need the following lemma.

\begin{lemma}\label{lem:stillinv}
Let $A$ be a matrix over $\Z[\pi]$ of the form $\id+P$ where $P$ is defined over $\Z_{\phi}^{>0}[\pi\doubler\cap \Z[\pi]$.
 Then there exists an open neighborhood $V'\subset H^1(N;\Q)$ of $\phi$ such that $A$ is also invertible over $\zpsig$ for any $\psi\in V'$.
\end{lemma}

\begin{proof}
Since $P$ is defined over $\Z[\pi]$ there are only finitely many group elements appearing in the entries of $P$.
It follows immediately that there exists
a neighborhood $V'$ of $\phi$ such that for any $\psi\in V'$ the matrix $P$ is in fact a matrix with entries in  $\Z_\psi^{>0}[\pi\doubler$.
It follows from Lemma~\ref{lem:inv} that $A=\id+P$ is invertible over $\Z_\psi^{>0}[\pi\doubler$ for any $\psi\in V'$.
\end{proof}

Let $V'$ be an open set containing $\phi$ as in Lemma~\ref{lem:stillinv} for the matrix 
\[ B_2=-\id+t\Big(\smfrac{\partial \psi'(e_i)}{\partial e_j}\Big)_{i,j=1,\dots,n}.\]
We  claim that $V\cap V'\subset H^1(N;\Q)$ has the required properties. Indeed, given any $\psi\in V\cap V'$ we have, by the discussion above,
that $-1+t$ and $B_2$  are invertible in $\zpsig$, hence $H_*(N';\Z_{\psi}[\pi\doubler)=0$ by Lemma~\ref{lem:turaev}.
This completes the proof of Theorem~\ref{thm:thurston}.

%

\section{Proof of Proposition~\ref{prop:surface-connected}}
\label{section:connected}
In the following recall that by a 
Thurston norm minimizing surface we mean a properly embedded oriented surface  $\Sigma$ in a 3-manifold $N$ which satisfies the following two conditions:
 \bn
\item $\chi_-(\Sigma)=\|\op{PD}([\Sigma])\|_T$,
\item there is no non-empty collection of components of $\Sigma$ that is (with the given orientation) null-homologous.
\en
Our goal is to prove the following proposition from the introduction.\\

\noindent \textbf{Proposition~\ref{prop:surface-connected}.}
\emph{Let $N$ be a 3-manifold and let $\phi\in H^1(N;\Z)$ be a primitive class. If $\tau(N,\phi)\ne 0$, then any Thurston norm minimizing surface dual to $\phi$ is connected.}
\\

In the proof of Proposition~\ref{prop:surface-connected} we will need the following lemma.

\begin{lemma}\label{lem:find-weight}
Let $N$ be a 3-manifold, let $\phi \in H^1(N;\Z)$ and let $\Sigma$ be a disconnected
Thurston norm minimizing surface dual to $\phi$. Then there exist components $\Sigma_1,\dots,\Sigma_k$ of $\Sigma$ and $n_1,\dots,n_k\in \Z$ such that the following conditions hold:
\bn
\item  $[\Sigma]=\sum_{i=1}^k n_i[\Sigma_i]$,
\item $N\sm \nu (\Sigma_1\cup \dots \cup \Sigma_k)$ is connected,
\item $\sum_{i=1}^k |n_i|>1$.
\en
\end{lemma}

The proof of the lemma makes very much use of the ideas of Turaev's that were employed in the proof of \cite[Lemma~1.2]{Tu02}.

\begin{proof}
Let $\Sigma$ be a disconnected
Thurston norm minimizing surface dual to $\phi$.  We denote its components by $\Sigma_1,\dots,\Sigma_r$. We define a \emph{weight} to be a function $w\colon \{1,\dots,r\}\to \Z_{\geq 0}$. Given a weight $w$ we define
\bn
\item $\Sigma(w)$ to be the disjoint union of $w(i)$-parallel copies of each $\Sigma_i$,
\item $\widehat{\Sigma(w)}$ to be the union of all $\Sigma_i$'s with $w(i)>0$,
\item $|w|:=\tmsum{i=1}{k} |w(i)|$,
\item $N(w):=\# \{i\in \{1,\dots,r\}\,|\, w(i)\ne 0\}$.
\en

\begin{claim}
Let $w$ be a weight with $b_0\big(N\sm \widehat{\Sigma(w)}\big)>1$, $[\Sigma(w)]=[\Sigma]$, $|w|>1$ and with $\chi_-(\Sigma(w))=\|\phi\|_T$. Then there exists a weight $v$ with $[\Sigma(v)]=[\Sigma]$,  $\chi_-(\Sigma(v))=\|\phi\|_T$, $|v|>1$ and such that $N(v)<N(w)$.
\end{claim}

Let $w$ be a weight with $b_0\big(N\sm \widehat{\Sigma(w)}\big)>1$, $[\Sigma(w)]=[\Sigma]$ and  $\chi_-(\Sigma(w))=\|\phi\|_T$. We pick a component $Y$ of $N(w)$. We denote by $\ol{Y}$ the closure of $Y$ in $N$. Then there exists a subset $I\subset \{1,\dots,r\}$ such that $w(i)>0$ for all $i\in I$ and such that the boundary of $\overline{Y}$ is the union of all $\Sigma_i, i\in I$  together with a subsurface of $\partial N$. For each $i\in I$ we denote by $\eps_i\in \{-1,1\}$ the sign such that for the oriented boundary we have $\partial \overline{Y}=\cup_{i\in I} \eps_i \Sigma_i$.
We define $I_+=\{ i\in I\,|\, \eps_i=1\}$ and similarly we define $I_-$.

It follows from the second condition on a Thurston norm minimizing surface, that $I_+\ne \emptyset$ and $I_-\ne\emptyset$.

We first consider the case that
\[ \smsum{i\in I_+}{} \chi_-(I_+)\,\,>\,\, \smsum{i\in I_-}{} \chi_-(I_-).\]
Then we define a new weight $v$ by
\[ v(j)\,\,=\,\, \left\{ \ba{rl} w(j),&\mbox{ if }j\not\in I,\\
w(j)-1,&\mbox{ if }j\in I_+,\\
w(j)+1,&\mbox{ if } j\in I_-.\ea\right.\]
It is clear that $[\Sigma(v)]=[\Sigma(w)]$ and $\chi_-(\Sigma(v))<\chi_-(\Sigma(w))$. But this contradicts the hypothesis
that $\chi_-(\Sigma(w))=\|\phi\|_T$. So this case can in fact not occur.
Similarly the  case that $\sum_{i\in I_-} \chi_-(I_-)> \sum_{i\in I_+} \chi_-(I_+)$ cannot occur.

Thus we only have to deal with the case that 
\[ \smsum{i\in I_+}{} \chi_-(I_+)\,\,=\,\, \smsum{i\in I_-}{} \chi_-(I_-).\]
We pick $i\in I$ such that $w(i)\leq w(j)$ for all $j\ne i\in I$. Without loss of generality we assume that $i\in I_+$. 
Then we define a new weight $v$ by
\[ v(j)\,\,=\,\, \left\{ \ba{rl} w(j),&\mbox{ if }j\not\in I,\\
w(j)-w(i),&\mbox{ if }j\in I_+,\\
w(j)+w(i),&\mbox{ if } j\in I_-.\ea\right.\]
It follows from
\[ \smsum{i\in I_+}{} \chi_-(I_+)\,\,=\,\, \smsum{i\in I_-}{} \chi_-(I_-)\]
that $\chi_-(\Sigma(v))=\chi_-(\Sigma(w))$. As above we have  $[\Sigma(v)]=[\Sigma(w)]$. Clearly we have $N(v)<N(w)$.
We just mentioned that $I_-\ne \emptyset$. So let $j\in I_-$. 
Then $|v|\geq |v(j)|=|w(j)+w(i)|\geq 2$. This concludes the proof of the claim.

Now we return to the actual proof of the claim.
If $N\sm \nu \Sigma$ is connected there is nothing to prove. Now suppose that is not the case. We consider the weight $w(i)=1$ for $i=1,\dots,r$. Since $\Sigma$ is disconnected we have $|w|=r>1$. 
We apply the claim iteratively till we end up with  a weight $v$ 
with $[\Sigma(v)]=[\Sigma(w)]=[\Sigma]$, $b_0\big(N\sm \widehat{\Sigma(v)}\big)=1$ and $|v|>1$. 
Then the components with $v(i)\ne 0$ and $n_i:=v(i)$, $i=1,\dots,r$ have the right properties.
\end{proof}

The following lemma is a special case of  \cite[Proposition~3.4]{FK06}.

\begin{lemma}\label{lem:fk06}
Let $N$ be a 3-manifold and let $\phi\in H^1(N;\Z)$ be non-zero. Let $F$ be a surface with components   $F_1,\dots,F_k$ and let $n_1,\dots,n_k\in \N$ such that the following hold:
\bn
\item $\op{PD}(\phi)=\sum_{i=1}^k n_i[F_i]$,
\item $N\sm \nu F$ is connected.
\en
If $\tau(N,\phi)\ne 0$, then there exists precisely one $i$ with $n_i\ne 0$.
\end{lemma}

\begin{proof}[Proof of Proposition~\ref{prop:surface-connected}]
Let $N$ be a 3-manifold and let $\phi\in H^1(N;\Z)$ be a primitive class with $\tau(N,\phi)\ne 0$.
Suppose there exists a disconnected  Thurston norm minimizing surface $\Sigma$ dual to $\phi$. By Lemma~\ref{lem:find-weight} there exist components $\Sigma_1,\dots,\Sigma_k$ of $\Sigma$ and $n_1,\dots,n_k\in \Z$ such that the following conditions hold:
\bn
\item  $[\Sigma]=\sum_{i=1}^k n_i[\Sigma_i]$,
\item $N\sm \nu (\Sigma_1\cup \dots \cup \Sigma_k)$ is connected,
\item $\sum_{i=1}^k |n_i|>1$.
\en
By Lemma~\ref{lem:fk06}, applied to $F=\Sigma_1\cup \dots \cup \Sigma_k$, we see that there exists precisely one $n_i\ne 0$. 
But from (3) it follows that $n_i>1$. But this shows that $\phi=n_i\op{PD}([\Sigma_i])$ is not primitive. Thus we obtained a contradiction.
\end{proof}

\section{Noncommutative torsion and degrees}
\label{section:torsion}

%
\subsection{The definition of noncommutative torsion}
\label{section:degrees}
Throughout this section let $\Gamma$ be a torsion-free elementary-amenable group.
Kropholler--Linnell--Moody \cite[Theorem~1.4]{KLM88} showed that the group ring $\Z[\Gamma]$ is  a domain.  Since $\Gamma$ is in particular amenable it follows from \cite[Corollary~6.3]{DLMSY03} that $\Z[\Gamma]$ is an Ore domain, which means that it has a classical ring of fractions $\K(\G)$. We refer to \cite[Section~8.2.1]{Lu02} for a helpful survey on Ore domains, Ore localizations and their properties.

Given a homomorphism $\phi\colon \Gamma\to \Q$ and a non-zero element 
$p=\sum_{g\in \Gamma} p_gg\in \Z[\Gamma]$ we write
\[ \deg_\phi(p)=\max\{ \phi(g)-\phi(h)\,|\,g,h\in \Gamma\mbox{ with }p_g\ne 0\mbox{ and }p_h\ne 0\}.\]
Since $\Z[\Gamma]$ is a domain this is in fact a homomorphism 
\[ \deg_\phi\colon(\Z[\Gamma]\sm \{0\},\,\,\cdot\,\,)\to (\Q_{\geq 0},+)\]
 of monoids. In particular we can extend it in the obvious way to a group homomorphism 
\[ \deg_\phi\colon (\K(\Gamma)^\times,\,\,\cdot\,\,) \to (\Q,+),\]
where $\K(\Gamma)^\times=\K(\Gamma)\sm \{0\}$. Since the target is commutative this gives rise to a homomorphism
\[ \deg_\phi\colon (\K(\Gamma)^\times_{\ab},\,\,\cdot\,\, ) \to (\Q,+),\]
where $\K(\Gamma)^\times_{\ab}$ denotes the abelianization of $\K(\Gamma)^\times$. 
Furthermore  we extend this definition to $0$ by setting $\deg_\phi(0)=-\infty$.

Now suppose that $N$ is a 3-manifold and $\phi\in H^1(N;\Q)$.
Let $\gamma\colon \pi_1(N)\to \G$ be a tfea-admissible homomorphism for $(N,\phi)$. If $H_*(N;\K(\Gamma))\ne 0$, then we set $\tau(N,\gamma)=0$. Otherwise we can consider the corresponding twisted Reidemeister torsion
\[ \tau(N,\gamma)\in \K(\Gamma)^\times_{\ab}\]
as defined in~\cite{Mi66,Tu01,Fr07}.
Furthermore we can consider the corresponding degree
\[ \deg_\phi(\tau(N,\gamma))\in \Q\cup \{-\infty\}.\]
A similar definition also applies to 2-complexes instead of 3-manifolds.

Now we recall a convenient method for calculating the twisted Reidemeister torsion. First of all, given a square matrix $A$ over $\K(\G)$ that is not invertible we set $\det(A)=0$.
Otherwise we denote by
$\det(A)\in \K(\G)^\times_{\ab}$ its Dieudonn\'e determinant, see e.g.\ \cite{Ros94} for the definition and properties.

Later on we will use the following lemma to calculate the Reidemeister torsion of a chain complex of length two which is essentially a special case of \cite[Theorem~2.2]{Tu01} and which also appeared as \cite[Lemma~6.2]{Fr07} and \cite[Lemma~6.6]{Fr07}.

\begin{lemma}\label{lem:torsion-2-complex}
If 
\[ 0\to \K(\Gamma)\xrightarrow{\bp c\,*\dots\,*\ep^t} \K(\Gamma)^n \xrightarrow{\bp *&*\\
*&B\ep} \K(\Gamma)^n \xrightarrow{\bp a\,*\dots\,*\ep}  \K(\Gamma)\to 0\]
is a chain complex of based right $\K(\G)$-modules with  $c\ne 0$ and $a\ne 0$,  then the corresponding Reidemeister torsion equals
\[ \det(B)\cdot c^{-1}\cdot a^{-1}.\]
Similarly, if 
\[0\to \K(\Gamma)^{n-1} \xrightarrow{\bp *\\
B\ep} \K(\Gamma)^n \xrightarrow{\bp a\,*\dots\,*\ep}  \K(\Gamma)\to 0\]
is a chain complex of based right $\K(\G)$-modules with $a\ne 0$,  then the corresponding Reidemeister torsion equals
\[ \det(B)\cdot  a^{-1}.\]
\end{lemma}

%
\subsection{The relationship to the Cochran-Harvey invariants}
\label{section:relation-to-ch}
In this section we recall the precise relationship of the degree functions of the last section to the original Cochran-Harvey invariants. 

First we recall the rational derived series of a group
$\pi$ introduced by Harvey~\cite[Section~3]{Ha05}. Let $\pi_r^{(0)}:=\pi$ and
define inductively
\[ \pi_r^{(n)}:=\big\{ g\in \pi_r^{(n-1)} | \, g^d \in \big[\pi_r^{(n-1)},\pi_r^{(n-1)}\big] \mbox{ for some }d\in \Z \sm \{0\} \big\}.\]
By \cite[Corollary~3.6]{Ha05} all the quotients $\pi/\pi_r^{(n)}$ are torsion-free elementary amenable groups. Thus for any 3-manifold $N$, any $\phi\in H^1(N;\Q)$ and any $n\geq 1$ the homomorphism  $\pi\to \pi/\pi_r^{(n)}$ is a tfea-admissible homomorphism for $(N,\phi)$.

Given a knot $K$ and $n\in \N_0$ Cochran~\cite{Co04} used the above quotients to introduce an integer invariant $\delta_n(K)$. Furthermore, given a 3-manifold $N$, a primitive class $\phi\in H^1(N;\Z)$ and $n\in \N_0$ Harvey~\cite{Ha05} also used the above quotients to introduce
an invariant $\delta_n(N,\phi)=\delta_n(\phi)$. 

The following two lemmas relate the Cochran--Harvey invariants to the degree invariants introduced in the previous section. Both lemmas are a consequence of Theorem~1.1 and Lemmas~4.3 and~4.4 in \cite{Fr07}.

\begin{lemma}\label{lem:relation-to-c}
Let $K\subset S^3$ be a knot. We denote by $X=S^3\sm \nu K$ the exterior of $K$, we write $\pi=\pi_1(X)$  and we pick a generator $\phi\in H^1(X;\Z)=\hom(\pi,\Z)$. Then for any $n\in \N_0$ we have
\[ \delta_n(K)\,\,=\,\,\deg_{\phi}\tau\big(X,\pi\to\pi/\pi_r^{(n+1)}\big)+\left\{\ba{ll} 1,&\mbox{ if }n=0,\\ 0,&\mbox{ otherwise.}\ea\right.\]
\end{lemma}

\begin{lemma}\label{lem:relation-to-h}
Let $N$ be a 3-manifold and let $\phi\in H^1(N;\Z)$ be primitive.
We write $\pi=\pi_1(N)$.  Then for any $n\in \N_0$ we have
\[ \delta_n(N,\phi)\,\,=\,\,\deg_{\phi}\tau\big(N,\pi\to\pi/\pi_r^{(n+1)}\big)+\left\{\ba{ll} 1+b_3(N),&\mbox{ if }n=0\mbox{ and }b_1(N)=1,\\ 0,&\mbox{ otherwise.}\ea\right.\]
\end{lemma}

%
\section{Proof of Theorem~\ref{thm:iso-intro}}\label{section:opentwo}
In this section we are going to prove Theorem~\ref{thm:iso-intro}.
We recall the statement of the theorem from the introduction.
\\

\noindent \textbf{Theorem  \ref{thm:iso-intro}.}\emph{ Let $N$ be a 3--manifold, let $\phi \in H^1(N;\Z)=\hom(\pi_1(N),\Z)$ be primitive and let $\gamma\colon \pi_1(N)\to \G$ be an admissible homomorphism. Assume that  $H_1(N;\Z_{\phi}[\G\doubler)=0$.
Then the following hold.
\bn
\item Every Thurston norm minimizing surface  dual to $\phi$ is connected.
\item For any  Thurston norm minimizing surface $\S$ dual to $\phi$
the inclusion induced maps
\[ \frac{\pi_1(\S)}{\ker(\gamma\colon \pi_1(\S)\to \G)}\xrightarrow{\,\,\i_\pm\,\,}
\frac{\pi_1(N\sm \nu \S)}{\ker(\gamma\colon \pi_1(N\sm \nu \S)\to \G)}  \]
are isomorphisms.
\item If $\G$ is a torsion-free elementary-amenable group, and if 
\[\deg_\phi(\tau(N,\gamma))=\|\phi\|_T,\]
then for any  Thurston norm minimizing surface $\S$ dual to $\phi$
the inclusion induced maps
\[ \frac{\pi_1(\S)}{\ker(\gamma\colon \pi_1(\S)\to \G)^{(1)}}\xrightarrow{\,\,\i_\pm\,\,}
\frac{\pi_1(N\sm \nu \S)}{\ker(\gamma\colon \pi_1(N\sm \nu \S)\to \G)^{(1)}}  \]
are isomorphisms.
\en}

\subsection{Proof of Theorem~\ref{thm:iso-intro} (1)}
Let $N$ be a 3--manifold, let $\phi \in H^1(N;\Z)=\hom(\pi_1(N),\Z)$ be primitive and let $\Sigma$ be a 
Thurston norm minimizing surface dual to $\phi$. Furthermore  let $\gamma\colon \pi_1(N)\to \G$ be an admissible homomorphism such  that  $H_1(N;\Z_{\phi}[\G\doubler)=0$. We write $\Gamma_\phi=\ker\{\phi\colon \Gamma\to \Z\}$. 
By Lemma~\ref{lem:h0123} we have
$H_*(N;\Z_{\phi}[\G\doubler)=0$.
It follows from Lemma~\ref{lem:novikov-zero} that $\tau(N,\phi)$ is monic, in particular non-zero.
Theorem~\ref{thm:iso-intro} (1) now follows from 
Proposition~\ref{prop:surface-connected}.

\subsection{Proof of Theorem~\ref{thm:iso-intro} (2)}
Throughout this section let $N$ be a 3--manifold, let $\phi \in H^1(N;\Z)=\hom(\pi_1(N),\Z)$ be primitive and furthermore let $\gamma\colon \pi_1(N)\to \G$ be an admissible homomorphism such  that  $H_1(N;\Z_{\phi}[\G\doubler)=0$. We write $\Gamma_\phi=\ker\{\phi\colon \Gamma\to \Z\}$. Let $\Sigma$ be any Thurston norm minimizing surface dual to $\phi$. By Theorem~\ref{thm:iso-intro} (1) we know that $\Sigma$ is connected.

Basically the same argument as in the proof of Theorem~\ref{thm:mainthmlow}
shows that 
\[ \i_-\colon \frac{\pi_1(\S)}{\ker(\gamma\colon \pi_1(\S)\to \G)}\,\,\to\,\,
\frac{\pi_1(N\sm \nu \S)}{\ker(\gamma\colon \pi_1(N\sm \nu \S)\to \G)}  \]
is an isomorphism. In the interest of space and readability we leave out the details.
By Lemma~\ref{lem:leftright} we also have 
$H_*(N;\Z_{-\phi}[\G\doubler)=0$. Repeating the same argument above shows that also the map induced by $\iota_+$ is an isomorphism.
This concludes the proof of  Theorem~\ref{thm:iso-intro} (2).

\subsection{Start of the proof of Theorem~\ref{thm:iso-intro} (3)}
Now assume that $\G$ is a torsion-free elementary-amenable group.
Let $\S$ be any connected Thurston norm minimizing surface dual to $\phi$. Denote by $g$ the genus of $\S$.
We pick once and for all a thickening $\S\times [-1,1]\subset N$ which is orientation preserving and as always we denote $\S\times (-1,1)$ by $\nu \S$.
We fix a base point $v$ on $\S\times \{-1\}\subset N$ and use it as a base point for $\S\times \{-1\}, N\sm \S\times (-1,1)$ and for $N$.
We give $\S\times \{1\}$ the corresponding base point which we denote by $v^+$.
We sometimes also refer to $v$ as $v^-$.
We also pick an embedded path $p$ in $N\sm \S\times (-1,1)$ from $v^+$ to $v^-$.
We denote by $t$ the element in $\pi_1(N,v)$ given by closing the above path $p$ by joining $v_-$ and $v_+$ by a constant path in $\S\times (-1,1)$. We give $t$ the orientation such that $\phi(t)=1$.

The choices of base points now allow us to define the twisted chain complexes
$C_*(\S;\zg), C_*(N\sm \nu \S;\zg)$ and $C_*(N;\zg)$.
We get an induced map 
\[ \i_-\colon C_*(\S;\zg)=C_*(\S\times \{-1\};\zg)\to C_*(N\sm \nu \S;\zg)\]
and the choice of the path $p$ allows us to define
an induced map 
\[\i_+\colon C_*(\S;\zg)=C_*(\S\times \{1\};\zg)\to C_*(N\sm \nu \S;\zg).\]
For both maps we refer to \cite[p.~933]{FK06} for details.
We will prove the following lemma.

\begin{lemma} \label{keylemma}
 Assume that  $H_1(N;\Z_{\phi}[\G\doubler)=0$ and $\deg_{\phi}(\tau(N,\gamma))=\|\phi\|_T$.
Then the inclusion induced maps
\[H_1(\S;\zg)\xrightarrow{\i_{\pm}} H_1(N\sm \S\times (-1,1);\zg)\] are isomorphisms.
\end{lemma}

The proof of Lemma~\ref{keylemma} will be given in the following three sections.
Before we turn to the proof of Lemma~\ref{keylemma} we  show how Theorem~\ref{thm:iso-intro} (3) follows from Lemma~\ref{keylemma}.
We write
\[ \ba{rcl}
A_{\pm}&=&\im\{\pi_1(\S)\xrightarrow{\i_\pm}\pi_1(N\sm \nu \S)\xrightarrow{\gamma}\G\} \\
B&=&\im\{\pi_1(N\sm \nu \S)\xrightarrow{\gamma}\G \}. \ea \]
We have a commutative diagram
\[ \xymatrix{ H_1(\S;\zg)\ar[r]^{\i_{\pm}}\ar[d]^{\cong} & H_1(N\sm\nu \S;\zg)\ar[d]^{\cong} \\
{\displaystyle \frac{\ker(\pi_1(\S)\to \G)}{\ker(\pi_1(\S)\to \G)^{(1)}}}\otimes \Z[\G/A_\pm]   \ar[r]^-{\i_\pm}&
{\displaystyle\frac{\ker(\pi_1(N\sm \nu \S)\to \G)}{\ker(\pi_1(N\sm \nu \S)\to \G)^{(1)}}}\otimes \Z[\G/B].  }\]
By Theorem~\ref{thm:iso-intro} (2) we know that the inclusion induced maps $\G/A_\pm \to \G/B$ are isomorphisms. It follows from Lemma~\ref{keylemma}
that the inclusion induced maps
\[ \frac{\ker(\pi_1(\S)\to \G)}{\ker(\pi_1(\S)\to \G)^{(1)}}\xrightarrow{\,\,\i_\pm\,\,}
\frac{\ker(\pi_1(N\sm \nu \S)\to \G)}{\ker(\pi_1(N\sm \nu \S)\to \G)^{(1)}}\]
are isomorphisms.
Now consider the following commutative diagram
\[ \xymatrix@C0.6cm{
1\ar[r]&\smfrac{\ker(\pi_1(\S)\to \G)}{\ker(\pi_1(\S)\to \G)^{(1)}} \ar[r]\ar[d]^{\i_\pm}&
\smfrac{\pi_1(\S)}{\ker(\pi_1(\S)\to \G)^{(1)}}  \ar[r]\ar[d]^{\i_\pm}&
\smfrac{\pi_1(\S)}{\ker(\pi_1(\S)\to \G)} \ar[d]^{\i_\pm}\ar[r]&1\\
1\ar[r]&\smfrac{\ker(\pi_1(N\sm \nu \S)\to \G)}{\ker(\pi_1(N\sm \nu \S)\to \G)^{(1)}} \ar[r]&
\smfrac{\pi_1(N\sm \nu \S)}{\ker(\pi_1(N\sm \nu \S)\to \G)^{(1)}}  \ar[r]&
\smfrac{\pi_1(N\sm \nu \S)}{\ker(\pi_1(N \sm \nu \S)\to \G)}\ar[r]&1.}
 \]
As we just pointed, the left vertical map is an isomorphism. The right vertical map is an isomorphism by Theorem~\ref{thm:iso-intro} (2), hence the middle map is an isomorphism as well.
This concludes the proof of Theorem~\ref{thm:iso-intro} (3) modulo the proof of Lemma~\ref{keylemma}.

\subsection{A chain complex calculating $\tau(N,\gamma)$}
Now we turn to the proof of Lemma~\ref{keylemma}.
We will only consider the case that $N$ has non-empty boundary.
The closed case is proved in a very similar fashion.

We retract $\S$ onto a one-dimensional CW-complex $S$ with one
0-cell $v$ and $2g$ 1--cells $e_1,\dots,e_{2g}$.
Furthermore we can retract $N\sm \nu \S$ to a two-dimensional CW-complex $Y$ 
with the following properties:
\bn
\item $Y$ contains $S\times \{-1\}$ and $S\times \{1\}$ as subcomplexes.
Henceforth, given a cell $c$ of $S$ we denote the corresponding cells of $S\times \{\pm 1\}$ by $c^\pm$,
\item $Y$ contains  $4g+1+r$ cells of dimension 1:  the $4g$ cells $e_1^-,\dots,e_{2g}^-$,$e_1^+,\dots,e_{2g}^+$,
one cell $t$ which connects $v^-$ and $v^+$ and corresponds to the chosen path connecting the base points, and finally
another $r$ cells $e_1',\dots,e_r'$ (where we can assume that the end point is $v^-$),
\item by an Euler characteristic argument $Y$ has $2g+r$ cells of dimension 2 which we denote by $f_1,\dots,f_{2g+r}$.
\en
Then we get a CW--complex $X$, that is homotopy equivalent to  $N$, by identifying $v^-=v^+$ and $e_i^-=e_i^+, i=1,\dots,2g$.
We get a Mayer--Vietoris type short exact sequence of chain complexes
\[ 0\to C_*(S;\Z[\Gamma])\xrightarrow{\i_--t\i_+} C_*(Y;\Z[\Gamma])\to C_*(X;\Z[\Gamma])\to 0,\]
we refer to \cite[Section~3]{FK06} for details.

The above short exact sequence of chain complexes translates into the
following commutative diagram:
\[ \xymatrix@R0.43cm{
0\ar[rr]\ar[ddd]&&
\oplus f_i\Z[\Gamma]  \ar[ddd]^-{{\tiny \bp R^- \\ R^+\\ R'\\ 0\ep}} \ar[rrrr]&&&&
\oplus f_i\Z[\Gamma]  \ar[ddd] \\
&&&&&&\\
&&&&&&\\
\oplus e_i\Z[\Gamma] \ar[rr]^-{{\tiny \bp \id \\ -\gamma(t)\cdot \id \\ 0\\0\ep}}\ar[dd]^{{\tiny \bp 1-\gamma(e_i)\ep}}&&
{\ba{r} \oplus e_i^-\Z[\Gamma]  \\ \oplus e_i^+\Z[\Gamma]
\\ \oplus e_i'\Z[\Gamma] \\ p \Z[\Gamma] \ea }
\ar[rrrr]^{\tiny \bp \gamma(t)\cdot \id &\id &0&0\\ 0&0&\id&0 \\ 0&0&0&\id \ep} \ar[dd]^-{\tiny \bp 1-\gamma(e_i)&0&0&\hfill 1 \\ 0&1-\gamma(e_i)&0&-1 \ep}&&&&
{\ba{r} \oplus e_i \Z[\Gamma]  \\ \oplus e_i'\Z[\Gamma] \\ p \Z[\Gamma] \ea }
 \ar[dd]^{\tiny \bp 1-\gamma(e_i)&0&1-\gamma(t)\ep } \\
&&&&&\\
 v\Z[\Gamma]\ar[rr]_{\tiny \bp 1\\ -\gamma(t)\ep}&&
{ \ba{r}v^- \Z[\Gamma]\\ v^+\Z[\Gamma] \ea }  \ar[rrrr]_{\tiny \bp \gamma(t)&1\ep}&&&&
v \Z[\Gamma].
}
\]
Here we omitted the $0$'s  to the left and right to save space. Also, by a slight abuse of notation we denote by $e_i\in \pi_1(S,p)$ the element represented by the one-cell $e_i$. 

Since $S$ is dual to $\phi$ it follows that $\phi$ also vanishes on $\pi_1(Y)$.
This implies in particular that the matrices $R^-,R^+,R'$ are defined over $\Z[\gphi]$.

\subsection{The matrices $R^-,R^+$ and $R'$}

In this section we will prove the following two lemmas.

\begin{lemma} \label{lem:rnondeg}
If $\deg_\phi(\tau(X,\g))=\|\phi\|_T=2g-1$, then $\bp R^-\\ R'\ep $ and $\bp R^+\\ R'\ep $ are invertible over $\K(\gphi)$.
\end{lemma}

\begin{lemma} \label{lem:rinv}
If $H_1(X;\zphig)=0$ and $\deg_\phi(\tau(X,\g))=\|\phi\|_T=2g-1$,
then $\bp R^-\\ R'\ep $ and $\bp R^+\\ R'\ep $ are invertible over $\Z[\gphi]$.
\end{lemma}

The proofs will occupy the remainder of this section.
First note that from the above commutative diagram it follows that the chain complex  $C_*(X;\K(\Gamma))$ is given by
\[
0\to
\oplus f_i\K(\Gamma)  \xrightarrow{\tiny \bp R^-+\gamma(t)R^+ \\ R'\\ 0\ep}
{\ba{r} \oplus e_i\K(\Gamma)  \\ \oplus e_i'\K(\Gamma) \\ p \K(\Gamma) \ea }
\xrightarrow{\tiny {\bp 1-\gamma(e_i)&0&1-\gamma(t) \ep}}
v\K(\Gamma) \to 0.\]
It follows from Lemma~\ref{lem:torsion-2-complex} that
\[ \tau(N,\g)=\det\bp R^-+\gamma(t)R^+\\ R'\ep \cdot (1-\gamma(t))^{-1}\in \K(\G)^\times_{\ab}\cup \{0\}.\]
Clearly $\deg_{\phi}(1-\gamma(t))=1$. It thus follows from our assumption $\deg_\phi(\tau(N,\g))=2g-1$ that 
\[ \deg_{\phi} \bigg(\det\bp R^-+\gamma(t)R^+\\ R'\ep\bigg)=2g.\]

Now Lemma~\ref{lem:rnondeg}  follows from a purely algebraic lemma.
To formulate the next lemma it is helpful to write $\K=\K(\G_\phi)$. As in~\cite{Fr07}  we can then identify $\K(\G)$ with a skew field of rational functions $\K(t)$ over $\K$. The function $\deg_\phi$ then translates into the usual notion of degree over $\K(t)$.
Now we are ready to state the  lemma that concludes
the proof of Lemma~\ref{lem:rnondeg}.

\begin{lemma} \label{lem:dim}
Let $A,B$ be two $k\times (k+l)$--matrices and $C$ an $l\times (k+l)$ matrix over the skew field $\K$. If
\[ \deg\,\det \bp A+tB\\ C\ep=k \]
then $\bp A\\ C\ep$ and $\bp B\\ C\ep$ are invertible over $\K$.
\end{lemma}

\begin{proof}
First we show that $C$ has rank $l$.
If not, then we can do row operations over $\K$  such that the bottom row of $\bp A+tB\\ C\ep$ becomes zero. But this implies that  $\bp A+tB\\ C\ep$ is not invertible over $\K(t)$. This contradicts the assumption that the degree of the determinant is finite.

Since $C$ has rank $l$ we can do column operations over $\K$ to turn $\bp A+tB\\ C\ep $ into a matrix of the form
\[  \bp A'+tB'&D'+tE' \\  0&C'\ep \]
for $k\times k$--matrices $A',B'$ over $\K$, $k\times l$--matrices $D',E'$ over $\K$ and an invertible $l\times l$--matrix $C'$ over $\K$.

Since $\deg(\det (C'))=0$  we have
\[ \deg\big(\det(A'+tB')\big)\,=\,\deg \det\bp A+tB\\ C\ep=k.\]
Since we did column operations over $\K$ it now suffices to show that $A'$ and $B'$ are invertible over $\K$.
Assume that the matrix $ A'$ is not invertible over $\K$. This implies that we can find an invertible $k\times k$--matrix $P$ over $\K$ such
that the first column of $A' P$ is zero.  Since the first column of $A'P$ is zero we can find an invertible $l\times l$--matrix $Q$ over $\K$ such that
\[ Q (A'+tB')P =\bp tb & v \\ 0&A''+tB''  \ep \]
where $b\in \K$ and $A'',B''$ are $(k-1)\times (k-1)$--matrices.
Since $\deg(tb)=0$  we have
\[ \deg\big( \det\bp A'+tB'\ep \big)\,=\,\deg \big(\det\bp A''+tB''\ep\big), \]
but by  \cite[Theorem~9.1]{Ha05} the latter is at most $k-1$. This contradiction shows that $A'$ is invertible over $\K$.

The proof that $B'$ is invertible over $\K$ is now virtually identical to the above proof.
\end{proof}

Now we turn to the proof of Lemma~\ref{lem:rinv}.
Henceforth we assume that we also have $H_1(X;\zphig)=0$.
By Lemma~\ref{lem:h0123}  we also have $H_*(X;\zphig)=0$ for all $i$.
We  consider $C_*(X;\zphig)$:
\[
0\to
\oplus f_i\zphig  \xrightarrow{\tiny \bp R^-+\gamma(t)R^+ \\ R'\\ 0\ep}
{\ba{r} \oplus e_i\zphig  \\ \oplus e_i'\zphig \\ p \zphig \ea }
\xrightarrow{\tiny {\bp 1-\gamma(e_i)&0&1-\gamma(t) \ep}}
v\zphig \to 0.\]
Since $1-\gamma(t)$ is invertible over $\zphig$ it follows from the obvious analogue of Lemma~\ref{lem:turaev} for chain complexes of length two and the fact that $C_*(X;\zphig)$ is acyclic, that 
\[ \bp R^-+\gamma(t)R^+ \\ R' \ep \] is invertible over $\zphig$.
By Lemma~\ref{lem:rnondeg} we already know that 
 ${\tiny \bp R^-\\ R'\ep}$  is invertible over $\K(\G_\phi)$, in particular the matrix is non--degenerate. It 
follows from Lemma~\ref{lem:inv} that $\tiny \bp R^-+\gamma(t)R^+ \\ R'\ep$ is invertible over $\zphig$ only if $\tiny \bp R^-\\R'\ep$ is invertible over $\zgphi$.

By Lemma~\ref{lem:leftright}
our assumption
$H_1(X;\zphig)=0$ implies that $H_1(X;\Z_{-\phi}[\Gamma\doubler)=0$. As above, 
 $C_*(X;\Z_{-\phi}[\Gamma\doubler)$ is acyclic which implies that 
$\tiny \bp R^-+\gamma(t)R^+ \\ R' \ep$ is invertible over $\Z_{-\phi}[\Gamma\doubler$.
But $ \tiny \bp R^-+\gamma(t)R^+ \\ R' \ep $ is invertible over $\Z_{-\phi}[\Gamma\doubler$  only if
$ \tiny \bp \g(t)^{-1}R^-+R^+ \\ R' \ep $ is invertible over $\Z_{-\phi}[\Gamma\doubler$.
By Lemma~\ref{lem:rnondeg}  we know that  $\tiny \bp R^+\\ R'\ep$ is non--degenerate. It
follows again from Lemma~\ref{lem:inv} that $\tiny \bp \g(t)^{-1}R^-+R^+ \\ R'\ep$ is invertible over $\Z_{-\phi}[\Gamma\doubler$ only if $\tiny \bp R^+\\R'\ep$ is invertible over $\zgphi$.
This concludes the proof of Lemma~\ref{lem:rinv}.

\subsection{Conclusion of the proof Lemma~\ref{keylemma}}

Now assume that  $H_1(X;\zphigright)=0$ (and hence $H_1(X;\Z_{-\phi}[\Gamma\doubler)=0$ by Lemma~\ref{lem:leftright}) and $\deg_\phi(\tau(X,\g))=\|\phi\|_T$.

Recall that we have to show that the inclusion induced maps
\[\i_\pm\colon H_1(\Sigma;\zg)\to H_1(N\sm \Sigma \times (-1,1);\zg)\]
 are isomorphisms.
Equivalently, we have to show that the inclusion induced maps
 \[\i_\pm\colon H_1(S;\zg)\to H_1(Y;\zg)\]
  are both isomorphisms.
First we will  show that $H_1(X;\Z_{-\phi}[\Gamma\doubler)=0$ together with $\deg_\phi(\tau(X;\K(\G)))=\|\phi\|_T$
implies that $\i_-\colon H_1(S;\zg)\to H_1(Y;\zg)$ is an isomorphism.

Note that $C_*(S;\zg)\xrightarrow{\i_-}C_*(Y;\zg)$ is given by
\[ \xymatrix{
0\ar[rrr]\ar[dd]&&&
\oplus f_i\Z[\Gamma]  \ar[dd]^-{\tiny \bp R^- \\ R^+\\ R'\\ 0\ep}\\
&&&\\
\oplus e_i\Z[\Gamma] \ar[rrr]^{\tiny \bp \id \\ 0 \\ 0\\0\ep}\ar[d]^{\bp \tiny 1-\gamma(e_i)\ep}&&&
{\ba{r} \oplus e_i^-\Z[\Gamma]  \\ \oplus e_i^+\Z[\Gamma]
\\ \oplus e_i'\Z[\Gamma] \\ p \Z[\Gamma] \ea }\ar[d]^-{\tiny \bp 1-\gamma(e_i)&0&0&\hfill 1 \\ 0&1-\gamma(e_{i})&0&-1 \ep} \\
 v\Z[\Gamma]\ar[rrr]_{\tiny \bp 1\\ 0\ep}&&&
{ \ba{r}v^- \Z[\Gamma]\\ v^+\Z[\Gamma]. \ea } 
}
\]
It  follows easily that $\i_-\colon H_1(S;\zg)\to H_1(Y;\zg)$ is an isomorphism if
\[ \oplus e_i^-\Z[\Gamma] \to \mbox{Coker}\left(  \bp R^- \\ R^+\\ R'\ep:\Z[\Gamma]\to \ba{r} \oplus e_i^-\Z[\Gamma]  \\ \oplus e_i^+\Z[\Gamma]
\\ \oplus e_i'\Z[\Gamma]  \ea  \right) \]
is an isomorphism. But this is clearly an isomorphism  since by Lemma~\ref{lem:rinv} the matrix $\bp R^+\\R'\ep$ is invertible.

An almost identical argument as above  shows that
 $H_1(X;\zphig)=0$ together with $\deg_\phi(\tau(X;\K(\G)))=\|\phi\|_T$
implies that $\i_+\colon H_1(S;\zg)\to H_1(Y;\zg)$ is an isomorphism.

This concludes the proof of Theorem~\ref{thm:iso-intro} (3) in the  case that $N$ has boundary.

The case that $N$ is closed is proved in a very similar fashion. The only  catch is that the diagrams get extended by one, which makes them even harder to navigate for the reader.

\end{document}